\documentclass[12pt,reqno]{amsart}
\usepackage{url}
\usepackage{amssymb}
\let\oldlabel=\label
\def\prellabel{\marginparsep=1em\marginparwidth=44pt
  \def\label##1{\oldlabel{##1}\ifmmode\else\ifinner\else
         \marginpar{{\footnotesize\ \\ \tt
                    ##1}}\fi\fi}}
\def\NZQ{\mathbb}               
\def\NN{{\NZQ N }}

\def\opn#1#2{\def#1{\operatorname{#2}}} 
\opn\chara{char}
\opn\rank{rank}
\opn\hilb{Hilb}
\opn\gr{gr}
\opn\Rees{{\mathcal R}}
\usepackage[latin1]{inputenc}
\usepackage{mathptmx}
\newtheorem{theorem}{Theorem}[section]
\newtheorem{lemma}[theorem]{Lemma}

\theoremstyle{definition}
\newtheorem{remark}[theorem]{Remark}

\newtheorem{remark/example}[theorem]{Remark/Example}

\newtheorem*{theorem*}{Theorem}

\let\epsilon=\varepsilon
\let\phi=\varphi
\let\kappa=\varkappa

\opn\ini{in}
\opn\KRS{KRS}
\opn\krs{krs}
\opn\Krs{Krs}
\opn\DEL{DEL}
\opn\diag{diag}
\opn\Ker{Ker}
\opn\Image{Im}
\opn\DD{{\mathcal D}}
\opn\SS{{\mathcal S}}
\opn\MM{{\mathcal M}}
\opn\GL{GL}
\opn{\hht}{ht}
\opn\Cl{Cl}
\opn\cl{cl}
\opn\height{height}
\opn\reg{reg}
\opn\Reg{Reg}
\opn\pd{pd}
\opn\supp{supp}
\opn{\HS}{HS}
\opn{\Ass}{Ass}
\opn{\Min}{Min}

\def\cI{{\mathcal I}}
\def\cP{{\mathcal P}}
\def\cQ{{\mathcal Q}}

\def\PP{{\mathcal A}}

\def\addots{\mathinner{\mkern1mu\raise1pt\hbox{.}\mkern2mu\raise4pt\hbox{.}
        \mkern2mu\raise7pt\vbox{\kern7pt\hbox{.}}\mkern1mu}}

\numberwithin{equation}{section}

\author{Aldo Conca}
\address{ Dipartimento di Matematica, Dipartimento di Eccellenza 2023-2027,
Universit\`a degli Studi di Genova, Italy}
\email{conca@dima.unige.it}
 
\title{A note on the $V$-invariant}


\date{}

\begin{document}

\begin{abstract}
Let $R$ be a finitely generated $\mathbb N$-graded algebra domain over a Noetherian ring and let $I$ be a  homogeneous ideal of $R$. Given $P\in \Ass(R/I)$ one defines the $v$-invariant $v_P(I)$ of $I$ at $P$ as the least $c\in \mathbb N$ such that $P=I:f$ for some  $f\in R_c$.   A classical result of Brodmann \cite{B} asserts that $\Ass(R/I^n)$ is constant for large $n$. So it makes sense to consider a prime ideal $P\in \Ass(R/I^n)$ for all the large $n$ and investigate how $v_P(I^n)$ depends on $n$. We prove that   $v_P(I^n)$  is eventually a  linear  function of $n$. 
When $R$ is the polynomial ring over a field this statement has been proved independently also by   Ficarra and Sgroi in their recent preprint \cite{FS}. 
  \end{abstract}

\maketitle


\section{The $v$-invariant}
Let $R$ be a finitely generated $\NN$-graded algebra domain over a Noetherian ring and let $I$ be a  non-zero  homogeneous ideal of $R$ generated by elements of degree $d_1,\dots,d_c\in \NN$. Recently there has been some interest in the study of an invariant associated to $I$ as follows. 
For $P\in \Ass(R/I)$ the $v$-invariant $v_P(I)$ of $I$ at $P$ is defined as 
$$v_P(I)=\min\{ u : \mbox{ there exist } f\in R_u \mbox{ such that } P=I:f\},$$
  see for example \cite{GRV}.  A classical result of Brodmann \cite{B} asserts that $\Ass(R/I^n)$ is constant for large $n$. 
Denote by $\PP(I)$ the set of the asymptotic associated primes of $I$, i.e. $\PP(I)=\Ass(R/I^n)$ for large $n$.  A quite natural question, discussed by Ficarra and Sgroi in the various arXiv versions of their paper \cite{FS}, is the nature of the asymptotic behaviour of $v_P(I^n)$ as a function on $n\in \NN$ for a prime ideal $P\in \PP(I)$.  The goal of this short note is to prove the following: 

\begin{theorem}
\label{main1}
For all $P\in \PP(I)$ the function $v_P(I^n)$ is eventually linear in $n$ with leading coefficient  in the set $\{d_1,\dots, d_c\}$. 
\end{theorem} 
  
 For the proof of  Theorem \ref{main1} we need the following characterization of $v_P(I)$ inspired by  \cite[3.2]{GRV}. 
  
  \begin{lemma}
\label{main2}
Given $P\in \Ass(R/I)$ let $X_P=\{ P_1\in \Ass(R/I) : P\subsetneq P_1\}$. 
Let $Q$ be the product of  the elements in $X_P$  with the convention that  $Q=R$ if $X_P$ is empty.  Then $v_P(I)$ is the smallest $w\in \NN$ such that the graded $R$-module $(I:P)/(I:(P+Q^\infty))$ does not vanish in degree $w$. 
  \end{lemma} 
  \begin{proof} 
  Let $v=v_P(I)$ and $w$ be the smallest natural number  such that $(I:P)/(I:(P+Q^\infty))$ does not vanish in degree $w$. Let $f\in R_v$ such that $I:f=P$. In particular $f\in I:P$. We show that  $f\not\in I:(P+Q^\infty)$ by contradiction. 
 If  $f\in  I:(P+Q^\infty)$ then  $fQ^m\subseteq  I$ for some $m\in \NN$ which implies $Q\subseteq  P$. Since $Q$ is the product of the ideals in $X_P$ (or $R$)  the inclusion   $Q\subseteq  P$ implies that  there exists $P_1\in X_P$  such that $P_1\subseteq P$, a contradiction because, by definition, $P$ is strictly contained in  $P_1$.  Hence $f\not\in I:(P+Q^\infty)$ and $v\geq w$. 
 
 Let now $g\in R_w$ such that  $g\in I:P$ and $g\not\in I:(P+Q^\infty)$. We have $P\subseteq I:g$. We claim that $P=I:g$ and this implies  $v\leq w$. We prove the claim by contradiction. Assume that $P\subsetneq I:g$. Observe  that the $I:g$ is a proper ideal since $g\not\in I$.  The multiplication with $g$   embeds $R/(I:g)$ in $R/I$. Hence $\Ass(R/I:g)\subseteq \Ass(R/I)$.  Then for each $P_1\in \Ass(R/I:g)$ we have 
 $$P\subsetneq I:g\subseteq  P_1\in \Ass(R/I:g)\subseteq \Ass(R/I)$$
 so that   $P_1\in X_P$. Hence $ \Ass(R/I:g) \subseteq X_P$.  Let $Q_1$ be the product of the primes in $ \Ass(R/I:g)$. For some large $m$ we have  $Q_1^m\subseteq I:g$, that is, $Q_1^mg\subseteq I$. Hence $Q^mg\subseteq Q_1^mg\subseteq I$, 
i.e.~$g\in I:(P+Q^\infty)$ contradicting the assumption. 
  \end{proof} 
  
  Now we are ready to prove Theorem \ref{main1} 
  
  \begin{proof}[Proof of Theorem \ref{main1}] Let $P\in \PP(I)$ and let $Q$ be the product of the ideals in $\{ P_1\in \PP(I)  : P\subsetneq P_1\}$. Consider the Rees ring $\Rees=\Rees(I)$ with its natural bigraded structure, i.e.~$\Rees_{(w,n)}=(I^n)_w$ for all $(w,n)\in \NN^2$.  Consider the extended ideals $\cI=I\Rees$ and $\cP=P\Rees$ and $\cQ=Q\Rees$ . Now set 
  $$A=\cI :_{\Rees}  \cP \quad \mbox{ and } \quad B=\cI:_{\Rees}  (\cP+\cQ^\infty).$$
   Set $H=A/B$. By construction $H$ is a finitely generated bigraded $\Rees$-module whose degree $(w,n)$ component equals the degree $w$ component of  
   $$( (I^{n+1}:P) \cap I^n)/ ( I^{n+1}:(P+Q^\infty)  \cap I^n ).$$
  By \cite[4.2]{R} for large $n$  one has  $I^{n+1}:I=I^n$ and so $I^{n+1}:P$ is contained in $I^n$. Therefore, for large $n$, the degree $(w,n)$ component of $H$ equals the degree $w$ component of  
  $$(I^{n+1}:P)/(I^{n+1}:(P+Q^\infty)).$$
   Hence by \ref{main2} we have $v_P(I^{n+1})=\min\{ w : H_{(w,n)}\neq 0\}$ for large $n$. 
  
  A simple variation of the arguments given in the proof of \cite[3.3]{BCV} or \cite[8.3.4]{BCRV} actually shows that the function $n\to \min\{ w : U_{(w,n)}\neq 0\}$ is eventually linear in $n$ with leading coefficient among the $d_i$'s for every  finitely  generated bigraded $\Rees$-module $U$. 
  
  Indeed, $\Rees$ can be presented as a quotient of polynomial extension $A=R[y_1,\dots, y_d]$  of $R$ 
  bigraded by $\deg y_i=(d_i,1)$  for all $i$ and $\deg x=(s, 0)$ for every $x\in R_s$. 
In \cite[3.3]{BCV} and  \cite[8.3.4]{BCRV} the base ring $A_0$ (which is $R$ is our context) is concentrated in degree $0$ so that the function  $n\to \max\{ w : U_{(w,n)}\neq 0\}$ is well defined for all finitely generated bigraded $A$-module $U$.  As we  deal with the function $n\to \min\{ w : U_{(w,n)}\neq 0\}$ the base ring can  be $\NN$-graded and rest of the arguments of   \cite[3.3]{BCV} and  \cite[8.3.4]{BCRV} apply to this context verbatim. 
  \end{proof} 
  
\begin{remark} If $J$ is a homogeneous reduction of $I$ then $\Rees(I)$ is finitely generated as an $\Rees(J)$-module, see \cite[Chap.8]{HS}.    Hence the module $H$ in the proof of Theorem \ref{main1} is finitely generated over $\Rees(J)$ and the leading coefficient of the function $n\to v_P(I^n)$ is equal to the degree of a generator of $J$. 
 \end{remark}

 \noindent {\bf Acknowledgements.} The author was supported  by the MIUR Excellence Department Project awarded to the Dept.~of Mathematics, Univ.~of Genova, CUP D33C23001110001, by PRIN 2020355B8Y  ``Squarefree Gr\"obner degenerations, special varieties and related topics'' and by  GNSAGA-INdAM.

\end{document}